\documentclass[letterpaper, 10 pt, conference]{ieeeconf}

\usepackage{amsmath,amssymb,euscript,yfonts,psfrag,latexsym,dsfont,graphicx}
\usepackage{bbm,color,amstext,wasysym,subfig,parskip,balance}
\graphicspath{{./},{./figures/}}

\newtheorem{theorem}{Theorem}

\newtheorem{problem}[theorem]{Problem}

\newcommand{\R}{{\mathbb R}}

\newcommand{\E}{{\mathbb E}}  

\newcommand{\tr}{\operatorname{trace}}

\definecolor{grey}{rgb}{0.6,0.6,0.6}
\definecolor{lightgray}{rgb}{0.97,.99,0.99}

\title{On local entropy, stochastic control and deep neural networks}
\author{Michele Pavon
\thanks{M.\ Pavon is with the Dipartimento di Matematica,
University of Padova, Padova, Italy; {email: pavon@math.unipd.it}}}
\begin{document}
\maketitle
\thispagestyle{empty}
\begin{abstract} In this paper, we connect some recent papers on smoothing of energy landscapes and scored-based generative models of machine learning to  classical work in stochastic control. We clarify these connections providing rigorous statements and representations which may serve as guidelines for further learning models.   
\end{abstract}

\section{Introduction} Wendell Fleming and co-workers and other scientists in  a series of papers between forty and thirty years ago \cite{Fle1, Fle2,Fle83,Fle3,P,DPP} showed that $-\log p$, where $p$ satisfies a Fokker-Planck equation, may be viewed as the value function of a stochastic control problem. In particular, in \cite[p.194]{P}, $-\log p(x,t)$ was named {\em local entropy} and various of its properties were established. The corresponding optimal control, see (\ref{OPTCONTR}) below, is then related to the so-called {\em score-function} $\nabla\log p(x,t)$ of generative models of machine learning based on flows \cite{ACTTV,RM,DSB,HCSDA,HJA,WKN,SSKKE,DTHD,WJXWY,VTLL,ZC,CLT}.  Local entropy was recently rediscovered in \cite{CCSY,COOSC} in connection with an attempt to smooth the {\em energy landscape} of deep neural networks. Differently from the work of Fleming and coworkers and of Chaudhari and coworkers, the stochastic control problems of which the local entropy is a value function are, as first observed in \cite{P}, {\em reverse-time} problems. We show in this paper that they are connected to a special case of a large deviation problem \cite{F2} first studied by Erwin Schr\"odinger in 1931-32 \cite{S1,S2} called {\em half-bridge problem}. The latter problem was recently used in \cite{PTT} for estimation of integrals through a variation of importance sampling. {We observe that \cite{P} is not concerned with Schr\"odinger bridges and \cite{PW}  deals with the full bridge problem. Hence, Theorem \ref{THM} connecting a reverse-time stochastic control problem to a half-bridge problem, appears new.  
This result is admittedly straightforward, but we hope  it might nevertheless be useful given a certain amount of confusion present in the machine learning literature on these topics, see Section \ref{discussion}}. 

The paper is outlined as follows. In Section \ref{DNN}, we provide some background on neural networks and on Gaussian smoothing of energy landscapes following \cite{COOSC}. In Section \ref{FED}, we collect a few basic results on finite-energy diffusions due to Edward Nelson and Hans F\"{o}llmer. In Section \ref{HBP}, we formulate two half bridge problems. We then provide a (reverse time) stochastic control formulation of one of them when the initial density is of the Boltzmann-Gibbs type. In Theorem \ref{THM}, we show that the local entropy is essentially the {\em value function} of the latter stochastic control problem. {Section \ref{discussion} contains some remarks on stochastic control in recent machine learning papers.}

\section{Background on deep neural networks}\label{DNN}
In some large dimensional problems, it is prohibitive to calculate the full gradient at each iteration. Consider for instance {\em deep neural networks}. A deep network has a multilayer architecture consisting of a nested composition of a linear transformation and a nonlinear one $\psi$ such as the sigmoid
\[\psi(x)=\frac{1}{1+\exp(-x)}
\]
or the {\em rectified linear unit} ReLU $\psi(x)=\max(x,0)$. In the learning phase of a deep network, one compares the predictions $y(x,\xi^i)$ for the input sample $\xi^i$ with the actual output $y^i$. This is done through a cost function $f_i(x)$, e.g.
\[f_i(x)=\|y^i-y(x;\xi^i)\|^2.
\]
The goal is to learn the {\em weights} $x\in\R^d$ through minimization of the empirical loss function
\[f(x)=\frac{1}{N}\sum_{i=1}^Nf_i(x).
\]
In modern datasets, such as ImageNet for image classification, $N$ can be in the millions. 
Therefore calculation of the full gradient $\frac{1}{N}\sum_{i=1}^N\nabla f_i(x)$  at each iteration to perform gradient descent is unfeasible.  One can then resort to {\em stochastic gradients} by sampling uniformly from the set $\{1,\ldots, N\}$ the index $i_k$ where to compute the gradient at iteration $k$
\begin{equation}\label{SGdescent}x_{k+1}=x_k-\eta_k\nabla f_{i_k}(x_{k-1})
\end{equation}
where $x^0$ is also random. In alternative, one can also average the gradient over a set of randomly chosen samples called  ``mini-batches", write $\nabla f_{{\rm mb}}(x)$. 
Assuming $\eta_k=\eta$ constant, one can rewrite (\ref{SGdescent}) as
\begin{eqnarray}\nonumber x_{k+1}-x_k&=&-\eta\nabla f(x_k)+\sqrt{\eta}V_k,\\ V_k&=&\sqrt{\eta}\left(\nabla f(x_k)-\nabla f_{i_k}(x_k)\right).\nonumber
\end{eqnarray}
Here the ``noise" $V_k$ has zero mean (i.e. $\nabla f_{i_k}(x_k)$ is an unbiased estimator of $\nabla f(x_k)$) and covariance  $\eta\Sigma(x_k)$ with

\[\Sigma(x)=\frac{1}{N}\sum_{i=1}^N\left(\nabla f(x)-\nabla f_i(x)\right)\left(\nabla f(x)-\nabla f_i(x)\right)^T.
\]
Under suitable assumptions \cite{LCW}, the above discrete iteration may be seen as discretization of a stochastic differential equation of the form
\[dX_t=-\nabla f(X_t)dt+\left[\beta^{-1}\Sigma(x)\right]^{1/2}dW_t, \quad \beta^{-1}=\eta,
\]
where $W$ is a standard $d$-dimensional Wiener process. Suppose $\Sigma(x)\equiv I$ the identity matrix\footnote{It has been observed \cite{CS,RJC} that the noise is actually quite far from being isotropic. It is conceivable to extend the results of this paper to this more general situation involving general diffusion processes using the results in \cite{ACC}.} so that the (SGD) equation is
\begin{equation}\label{SGD}
dX_t=-\nabla f(X_t)dt+\beta^{-1/2}dW_t,
\end{equation}
Then, if $\exp [-\beta f]$ is integrable, the distribution of $X_t$ tends to the invariant measure of the Boltzmann-Gibbs type
\begin{equation}\label{BG}\bar{\rho}(x;\beta)=Z(\beta)^{-1}\exp[-\beta f(x)].
\end{equation}
It is also well-known that, as the inverse temperature parameter $\beta$ increases to infinity, the stationary measure tends to concentrate  on the absolute minima of $f$.

The loss function, besides being non convex, provides an extremely rugged {\em energy landscape}. Local minima that generalize well lie in wide valleys where most of the eigenvalues of the Hessian are close to zero\footnote{These regions are robust with respect to data perturbations, noise in the activations and perturbation of the parameters, cf. \cite{SBL}, \cite[Section 1]{CCSY}.}. It is therefore useful to {\em smooth} the loss function. One possible way to do this is via {\em local entropy} \cite{CCSY,COOSC} which was motivated by \cite{BILSZ} which studies energy landscapes for the discrete perceptrons. Using their notation, the latter is defined as the convolution of $\bar{\rho}$ with the heat kernel
\begin{subequations}\label{LocEnt}
\begin{eqnarray}\nonumber
f_\gamma(x)&=&u(x,\gamma)\\&=&-\frac{1}{\beta}\log\left(G_{\beta^{-1}\gamma}\star\exp(-\beta f(x)\right)\\
G_\gamma(x)&=&(2\pi\gamma)^{-\frac{d}{2}}\exp\left(-\frac{||x\|^2}{2\gamma}\right).
\end{eqnarray}
\end{subequations}
It follows immediately, \cite[Lemma 2]{COOSC}, that $u(x,\gamma)$ satisfies
\begin{eqnarray}\label{VHJa}
&&\frac{\partial u}{\partial t}=-\frac{1}{2}\|\nabla u\|^2+\frac{1}{2\beta}\Delta u, \quad 0\le t\le \gamma,\\&&u(x,0)=f(x).
\label{VHJb}\end{eqnarray}
It is also immediate that if we define
\[\rho(x,t):=Z(\beta)^{-1}\exp[-\beta u(x,t)],
\]
then $\rho$ satisfies 
\begin{eqnarray}\label{HEATa}
&&\frac{\partial \rho}{\partial t}=\frac{1}{2\beta}\Delta \rho, \quad 0\le t\le \gamma,\\&&\rho(x,0)=\bar{\rho}(x).
\label{HEATb}\end{eqnarray}
This is just the Fokker-Planck equation for the stochastic process
\begin{equation}\label{WienerProcess}dZ_t=\beta^{-1/2}dW_t, \quad Z_0\sim\bar{\rho}(x)dx.
\end{equation}
Hence, 
\begin{equation}\label{LE}u(x,t)=-\frac{1}{\beta}\log\rho(x,t)+ c(\beta).
\end{equation}
Obviously, $\bar{\rho}$, which is invariant for (\ref{SGD}), is not invariant for (\ref{WienerProcess}) and the energy landscape gets smoothed. In order to unveil the connection between local entropy and a certain large deviation problem, we shall need to consider reverse-time stochastic control problems in the spirit of \cite{P}.

\section{Backgound on finite-energy diffusions}\label{FED}

Consider a physical system consisting of a large number $N$ of particles. Their flow is described over the time interval $[0,1]$ by a Fokker-Planck equation. Suppose that at the final time $1$ the probability density is found to be approximatively equal to $\rho_1(x)$ which is, however, not compatible with the ``a priori" evolution.  If we are confident in our reference model, we see that something ``exotic" has occurred. In the spirit of Boltzmann \cite{BOL} and Schr\"odinger \cite{S1,S2}, we can then pose the following question: Of the many unlikely ways in which this may have occurred, which one is the most likely? This is a problem of {\em large deviations of the empirical distribution} which, thanks to Sanov's theorem \cite{SANOV}, is a equivalent to a maximum entropy problems for measures on paths \cite{F2}.
As observed in \cite{PTT}, solving this problem allows to reconstruct the past evolution of a system from the final marginal and the reference evolution (without the latter, the problem is typically ill-posed). This may be viewed as a generalization of the Bayesian paradigm, featuring a large number of applications in many fields of science. 

Schr\"odinger's original hot gas {\em Gedankenexperiment} featured $N$ independent Brownian particles and estimates of the marginals both at time $t=0$ and at time $t=1$. Besides its original statistical mechanics motivation, its importance lies with the inference method (see above) and as a computationally attractive regularization of the important Optimal Mass Transport problem, see \cite{W,Leo,PC,CGP21,CGPAR} for  survey papers.    As we shall see, the solution of the ``half-bridge" problem described above is much simpler than that for the full bridge problem.

In order to provide the modern formulation of this problem, we first recall a few basic facts about the kinematics of finite energy diffusion due essentially to Nelson \cite{N1} and F\"ollmer \cite{F1,F2}.

Let $\Omega:=C([0,1],\R^d)$ denote the family of $d$-dimensional continuous functions, $W_x$ denote Wiener measure on $\Omega$ starting at $x$ at $t=0$. If, instead of a Dirac measure concentrated at $x$, we give the volume measure as initial condition, we get the unbounded\footnote{Therefore, $W$ is not a probability measure. Its marginals at each point in time coincide with the volume measure.} measure
\begin{equation}\label{SWM}W:=\int W_x\,dx
\end{equation}
on path space, which is called {\em stationary Wiener measure} (or, sometimes, reversible Brownian motion). It is a useful tool to introduce the family of distributions $\mathcal D$ on $\Omega$ which are equivalent to the scaled measure $W_{\sigma^2}$, namely Wiener measure with variance $\sigma^2 I_d$.
By Girsanov's theorem {\cite{KS}}, under $Q\in\mathcal D$, the coordinate process $x(t,\omega)=\omega(t)$ admits the representations
\begin{eqnarray}\nonumber
dX(t)&=&\beta_+ dt+ \sigma dW_+(t), \quad \beta_+ \;{\rm is}\; {\cal F}_t^- - {\rm adapted},\\
dX(t)&=&\beta_- dt+ \sigma dW_-(t), \quad \beta_- \;{\rm is}\; {\cal F}_t^+ - {\rm adapted},\nonumber
\end{eqnarray}
where ${\cal F}_t^-$ and ${\cal F}_t^+$ are $\sigma$- algebras of events observable up to time $t$ and from time $t$ on, respectively, and $W_-$, $W_+$ are standard $d$-dimensional Wiener processes\footnote{In \cite{CLT}, the forward and backward stochastic differentials in $(1)$ and $(2)$ feature the same Wiener process  which is impossible excepting trivial cases. The statement a couple of lines below $(2)$ that ``these two stochastic processes are equivalent in the sense that their marginal densities are equal to each other throughout $t\in [0,T]$; in other words, $p_t^{(1)}\equiv p_t^{(2)}$" is correct but incomplete. Indeed, these two processes induce the same measure on path space and have, in particular, equal finite-dimensional distributions not just the one-time densities.} \cite{F1}. Moreover, the forward and the backward drifts $\beta_+, \beta_-$ satisfy the finite-energy condition
$$Q\left[\int_{0}^{1}\|\beta_+\|^2dt <\infty\right]=Q\left[\int_{0}^{1}\|\beta_-\|^2dt <\infty\right]=1.
$$
They can be obtained as Nelson's conditional derivatives \cite{F1}
\begin{eqnarray}\nonumber\beta_+(t)&=&\lim_{h\searrow 0}\frac{1}{h}\E\left[X(t+h)-X(t)|{\cal F}_t^-\right],\\\beta_-(t)&=&\lim_{h\searrow 0}\frac{1}{h}\E\left[X(t)-X(t-h)|{\cal F}_t^+\right],\nonumber
\end{eqnarray}
where the limits are in $L_d^2(\Omega,\mathcal F, Q)$\footnote{$L_d^2(\Omega,\mathcal F, Q)$ is the Hilbert space of (equivalence classes) of $d$-dimensional, real random vectors possessing finite second moment, i.e. $\tr\left[\E\left(XX^T\right)\right]<\infty$.}. It was also shown in \cite{F1} that the one-time
probability density $\rho(\cdot,t)$ of $X(t)$  exists for every $t\in[0,1]$ and the following relation holds a.s.
\begin{equation}\label{DUALITY} \E\{\beta_+(t)-\beta_-(t)|x(t)\} =
\sigma^2\nabla\log\rho(X(t),t). 
\end{equation} 
The finite-energy diffusion $\{X(t); 0\le t\le 1\}$ is called {\it Markovian} if there
exist two measurable functions $b_+(\cdot,\cdot)$ and $b_-(\cdot,\cdot)$ such
that  $\beta_+(t)=b_+(X(t),t)$ a.s. and
$\beta_-(t)=b_-(X(t),t)$ a.s., for all $t$ in $[0,1]$. The duality relation
(\ref{DUALITY})  is now implied by Nelson's relation \cite{N1} 
\begin{equation}\label{NELSONDUALITY}  b_+(x,t)-b_-(x,t) =
\sigma^2\nabla\log\rho(x,t). 
\end{equation} 
Observe that stationary Wiener measure is a Markovian measure with $b_+(x)=b_-(x)\equiv 0$, since the density of the invariant Lebesgue measure is $\rho(x,t)\equiv 1, \forall t$.

For $Q,P\in\mathcal D$, the {\em relative entropy} ({\em Divergence}, {\em Kullback-Leibler index}) $H(Q,P)$ of $Q$ with respect to $P$ is
$$H(Q,P)=\E_Q\left[\log\frac{dQ}{dP}\right].$$
It then follows from Girsanov's theorem  \cite{KS} that
\begin{subequations}
\begin{eqnarray}\nonumber
H(Q,P)&=&H(q_0,p_0)\\&+&\E_Q\left[\int_{0}^{1}\frac{1}{2\sigma^2}
\|\beta_+^Q-\beta_+^P\|^2dt\right]\label{RE1}\\\nonumber
&=&H(q_1,p_1)\\&+&\E_Q
\left[\int_{0}^{1}\frac{1}{2\sigma^2}
\|\beta_-^Q-\beta_-^P\|^2dt\right].\label{RE2}
\end{eqnarray}
\end{subequations}
Here $q_0$, $q_1$ ($p_0$, $p_1$) are the marginal {distributions} of $Q$ ($P$) at $0$ and $1$, respectively. Moreover, $\beta_+^Q$ and
$\beta_-^Q$ are the
forward and the backward drifts of $Q$, respectively, and similarly for $P$.

\section{Half bridge Problems}\label{HBP}
Let $W_{\sigma^2}$ be scaled stationary Wiener measure and let $\rho_0$ and $\rho_1$ be probability densities. Let $\mathcal D(\rho_0)$ and $\mathcal D(\rho_1)$ denote the set of distributions in $\mathcal D$ having the prescribed marginal at the initial and final time, respectively.
We consider the following two maximum entropy problems:
\vspace{0.5cm}
\begin{problem}\label{halfbridge1}
\begin{equation}{\rm Minimize}\quad H(Q,W_{\sigma^2}) \quad {\rm over} \quad Q\in{\cal D}(\rho_0),
\end{equation} 
\end{problem}
\begin{problem}\label{halfbridge2}
\begin{equation}{\rm Minimize}\quad H(Q,W_{\sigma^2}) \quad {\rm over} \quad Q\in{\cal D}(\rho_1),
\end{equation} 
\end{problem}
The large deviations motivation for studying these two problems is as follows: we have observed an initial marginal (resp. final marginal) for a large number of independent Brownian particles which does not agree with our prior model $W_{\sigma^2}$. What is the most likely evolution for these particles? Again, by Sanov's theorem \cite{SANOV}, the solution is obtained solving the maximum entropy Problems \ref{halfbridge1} (\ref{halfbridge2}). The solution of these problems follows immediately from representations (\ref{RE1})-(\ref{RE2}). Indeed, consider Problem \ref{halfbridge1}. Since $H(q_0,p_0)=H(\rho_0,p_0)$ is constant over ${\cal D}(\rho_0)$, it follows that it is optimal to make the Lagrangian cost equal to zero, namely the optimal $Q_1^\star$ has $\beta_+^{Q_1^*}=\beta_+^W=0$. We conclude that the optimal measure $Q_1^\star$ is induced on trajectories by the process
\begin{equation}\label{optimal1}
X_{1}(t)=X_{1}(0)+ \sigma W(t), \quad X_{1}(0)\sim \rho_0(x)dx,
\end{equation}
with $W$ standard $n$-dimensional Wiener process. Basically, to solve such a problem, it suffices to change the initial distribution of stationary Wiener measure. Similarly, to solve Problem \ref{halfbridge2}. Since $H(q_1,p_1)=H(\rho_1,p_1)$ is constant over ${\cal D}(\rho_1)$, it follows that  the optimal $Q_2^\star$ has $\beta_-^{Q_2^*}=\beta_-^W=0$. We conclude that the optimal measure $Q_2^\star$ is induced on trajectories by the process
\begin{equation}\label{optimal1}
X_{2}(t)=X_{2}(1)+ \sigma \bar{W}(t), \quad X_{2}(1)\sim \rho_1(x)dx,
\end{equation}
with $\bar{W}$ standard $d$-dimensional Wiener process with $\bar{W}(1)=0$. To solve such a problem, it namely suffices to change the final distribution of stationary Wiener measure. Let us now find the backward drift of $Q^\star_1$ and the forward drift of $Q^\star_2$. Let $q^\star_1(x,t)$ ($q^\star_2(x,t)$) be the probability density of $X_1(t)$ ($X_2(t)$). From (\ref{NELSONDUALITY}) (prior and therefore solutions are Markovian), we get
\begin{eqnarray}\label{FBDRIFT1}
\beta_-^{Q_1^*}&=&-\sigma^2\nabla \log q^\star_1(X_1(t),t),\\ \beta_+^{Q_2^*}&=&\sigma^2\nabla \log q^\star_2(X_2(t),t).
\label{FBDRIFT2}
\end{eqnarray}
We are interested in providing a stochastic control characterization of the $Q^\star_1$ in the case when the initial density in Problem \ref{halfbridge1} is $\rho_0(x)=\bar{\rho}(x)$ given by (\ref{BG}) and $\sigma^2=1/\beta$. Let $\bar{W}$ be as in (\ref{optimal1}) and $\mathcal V$ be the family of finite-energy control functions adapted to the future of $\bar{W}$. Then we consider 
\begin{problem}\label{RTSC1}
\begin{eqnarray}\nonumber&&{\rm Minimize}_{v\in\mathcal V}\quad \E\left\{\int_0^t \frac{\beta}{2}\|v\|^2dt-\log\bar{\rho}(X_0)\right\}\\&&{\rm subject \; to}\nonumber\\&&dX_s=vds+\beta^{-1/2}d\bar{W}_s, \quad X_t=x.\nonumber
\end{eqnarray}
\end{problem}
We are now ready to unveil the connection between the local entropy $u(x,t)$ defined in (\ref{LocEnt}) and stochastic control. 
\begin{theorem}\label{THM} Consider Problem \ref{RTSC1}. The value function 
\[V(x,t):=\inf_{v\in\mathcal V}\E\left\{\int_0^t \frac{\beta}{2}\|v\|^2dt-\log\bar{\rho}(X_0)+c(\beta)|X_t=x\right\}
\]
is related to the local entropy (\ref{LocEnt}) by 
\begin{equation}\label{relation}\frac{1}{\beta}V(x,t)=u(x,t).
\end{equation}
The optimal control is of the feedback ({\em score function}) type given by
\begin{equation}\label{OPTCONTR}v^\star(x,t)=\frac{1}{\beta}\nabla V(x,t)=-\frac{1}{\beta}\nabla\log\rho(x,t).
\end{equation}
The measure induced on the trajectories is just $(Q_1^\star)_{tx}=Q_1^\star\left[\,\cdot\mid X_t=x\right]$.
\end{theorem}
\begin{proof}
Indeed, if $V$ is sufficiently regular, it satisfies the Backward Dynamic programming equation \cite{FR}
\begin{equation}\label{DPE}\frac{\partial V}{\partial t}-\frac{1}{2\beta}\Delta V=\min_{v\in\R^n}\left[-v\cdot\nabla V +\frac{\beta}{2}\|v\|^2\right]=-\frac{1}{2\beta}\|\nabla V\|^2,
\end{equation}
with boundary condition $V(x,0)=-\frac{1}{\beta}\log\bar{\rho}(x) + c(\beta)$. Multiplying (\ref{DPE}) by $1/\beta$, we get that $(1/\beta)V(x,t)$ satisfies the same equation as the local entropy $u(x,t)$. Considering the initial condition, we see that $(1/\beta)V(x,t)$ and $u(x,t)$ differ by the (irrelevant) additive constant $c(\beta)$. Moreover, the optimal backward drift is $v^*(x,t)=-\frac{1}{\beta}\nabla\log\rho(x,t)$ where $\rho$ satisfies (\ref{HEATa})-(\ref{HEATb}). This is the same backward drift (\ref{FBDRIFT1}) of $Q_1^*$ ($\sigma^2=1/\beta$) if we im such as \cite{}pose as initial marginal $\bar{\rho}$. Thus, the measure corresponding to the optimal control is just $(Q_1^*)_{tx}$.
\end{proof}
This result, as far as the stochastic control representation goes, is connected to \cite[Theorem 2.1]{P} (case $f=0, b_+=0$ and $a=1/\beta$). The connection to Schr\"odinger Bridges can also be made starting from the observation that $\rho(x,t)$ satisfying (\ref{HEATa})-(\ref{HEATb}) is {\em space-time harmonic} along the lines outlined at the end of Section 3 in \cite{PW}.
Finally, an alternative proof can be based on a ``completion of  squares argument"  \cite[p.679-680]{CGP4}
\section{Discussion}\label{discussion}

Theorem \ref{THM} provides the correct variational representation for the local entropy $u(t,x)$ through (\ref{relation}). This variational representation may be used as a basis of numerical schemes seeking the minima of the local entropy.

 In  \cite{COOSC}, trying to connect the local entropy to a standard stochastic control problem with an initial condition for the controlled evolution, see (CSGD) there, and a final cost in the criterion, see (24) there, led to some confusion in an otherwise quite interesting paper. It is there stated in the abstract that ``A stochastic control interpretation is used to prove that a modified algorithm converges faster than the SGD in expectation." Theorem 12 in Section 5 is claimed to provide theoretical support to the fact that ``minimizing local entropy leads to an improvement in the original loss $f(x)$ as compared to stochastic gradient descent".  It seems to the present author that  Theorem 12 amounts to the observation that the zero control does no better than the optimal control, a statement which, evidently, requires no mathematical proof.

In closing, we mention that it would be interesting to connect these reverse-time stochastic control problems to the reverse process occurring in various generative models of machine learning such as \cite{HJA,ZC,CLT,SSKKE}.

\end{document}